\documentclass[11 pt]{amsart}
\usepackage[utf8x]{inputenc}
\usepackage{amssymb,amsmath,amsfonts,mathrsfs}
\usepackage{appendix}
\newtheorem{theorem}{Theorem}[section]
\newtheorem{lemma}[theorem]{Lemma}
\newtheorem{cor}[theorem]{Corollary}

\theoremstyle{definition}

\theoremstyle{remark}
\newtheorem{remark}[theorem]{Remark}

\numberwithin{equation}{section}

\newcommand{\loc}{\operatorname{loc}}
\newcommand{\comp}{\operatorname{comp}}

\newcommand{\lloc}{L_{\loc}}

\newcommand{\Dom}{\operatorname{Dom}}
\newcommand{\End}{\operatorname{End}}

\newcommand{\supp}{\operatorname{supp}}

\newcommand{\RE}{\operatorname{Re}}

\newcommand\RR{\mathbb{R}}

\newcommand{\mcomp}{C_{c}^{\infty}(M)}

\newcommand{\Del}{\Delta}

\newcommand\esssup{\rm ess\,\sup}

\newcommand{\vbn}{\mathcal{V}}
\newcommand{\vcomp}{C_{c}^{\infty}(\mathcal{V})}
\newcommand{\vcompr}{C_{c}^{\infty}(\mathcal{V}|_{B_{\rho}})}
\newcommand{\bo}{\nabla^{\dagger}\nabla}

\begin{document}

\title[covariant Schr\"odinger operators]{Self-adjointness of non-semibounded covariant Schr\"odinger operators on Riemannian manifolds}
\author{Ognjen Milatovic}
\address{Department of Mathematics and Statistics\\
         University of North Florida   \\
       Jacksonville, FL 32224 \\
        USA
           }
\email{omilatov@unf.edu}


\begin{abstract}
In the context of a geodesically complete Riemannian manifold $M$, we study the self-adjointness of $\nabla^{\dagger}\nabla+V$ where $\nabla$ is a metric covariant derivative (with formal adjoint $\nabla^{\dagger}$) on a Hermitian vector bundle $\vbn$ over $M$, and $V$ is a locally square integrable section of $\End\vbn$  such that the (fiberwise) norm of the ``negative" part $V^{-}$ belongs to the local Kato class (or, more generally, local contractive Dynkin class). Instead of the lower semiboundedness hypothesis, we assume that there exists a number $\varepsilon \in [0,1]$ and a positive function $q$ on $M$ satisfying certain growth conditions, such that $\varepsilon \nabla^{\dagger}\nabla+V\geq -q$, the inequality being understood in the quadratic form sense over $\vcomp$. In the first result, which pertains to the case $\varepsilon \in [0,1)$, we use the elliptic equation method. In the second result, which pertains to the case $\varepsilon=1$, we use the hyperbolic equation method.
\end{abstract}
\maketitle                   






\section{Introduction}\label{S:intro}
\subsection{Motivation of the problem} One of the central objects of quantum mechanics is a Schr\"odinger differential expression
\begin{equation}\label{E:general-op}
S_{V}:=\Delta +V,
\end{equation}
where $\Delta$ is the (non-negative) Laplacian on a Riemannian manifold $M$ with a metric $g$, and $V$ is a locally square integrable real-valued function on $M$. Let $\mcomp$ denote smooth compactly supported functions on $M$ and let $L^2(M)$ be the space of square integrable functions on $M$. In the quantum mechanical setting, a self-adjoint extension of $S_{V}|_{\mcomp}$ in $L^2(M)$ is referred to as a \emph{quantum observable} associated with $S_{V}$. In view of lemma~\ref{L:GT-lemma} (see appendix~A of this paper), a self-adjoint extension of $S_{V}|_{\mcomp}$ in $L^2(M)$ always exists, although its uniqueness is not guaranteed (without additional conditions). The \emph{essential self-adjointness} of $S_{V}|_{\mcomp}$ means that $S_{V}|_{\mcomp}$ has a unique self-adjoint extension in $L^2(M)$. Furthermore, it is known (see section VI.1.7 in~\cite{ber-book}) that a self-adjoint extension of $S_{V}|_{\mcomp}$ is unique if and only if for all $u_0\in\mcomp$, the Cauchy problem
\begin{equation}\nonumber
\frac{1}{i}\frac{\partial u(t)}{\partial t}=S_{V}u(t),\quad u(0)=u_0
\end{equation}
has a unique solution $u(t)\in L^2(M)$, for all $t\in\mathbb{R}$. (Here, the $t$-derivative is taken in the $L^2$-norm sense, and $S_{V}u$ is understood in the sense of distributions.) In this case, the corresponding quantum system is said to have the \emph{quantum completeness} or \emph{quantum confinement} property.

Switching to the classical mechanics viewpoint, on the cotangent bundle $T^*M$ (with the usual symplectic structure) we consider the Hamiltonian
\begin{equation}\label{E:hamiltonian}
h(p,x)=|p|^2+V(x),
\end{equation}
where $p\in T_{x}^*M$ and $|p|$ stands for the length of $p$ in the metric on $T_{x}^*M$ induced by $g$. In the coordinates $(x^1,\dots, x^n,p^1,\dots, p^n)$, where $\dim M=n$, the corresponding Hamiltonian system is
\begin{equation}\label{E:hamilton-sys}
\frac{d x^j}{d t}=\frac{\partial h}{\partial p^j}, \quad \frac{d p^j}{d t}=-\frac{\partial h}{\partial x^j},\quad j=1,\dots, n.
\end{equation}

Assuming that $V\in C^2(M)$, the local Hamiltonian flow corresponding to~(\ref{E:hamiltonian}) is well defined. If the solutions to~(\ref{E:hamilton-sys}) with arbitrary initial data are defined for all $t$, then the corresponding classical system is referred to as \emph{classically complete.} As seen from the examples in~\cite{rauch-reed}, even in the case $M=\mathbb{R}$ the classical and quantum completeness notions are independent. For a further discussion of the quantum and classical completeness, see~\cite{rs}.

\subsection{A brief look at the existing literature}\label{SS:cursory}
In the context of $M=\mathbb{R}^n$, the essential self-adjointness problem for $-\Delta +V$ (here, $\Delta$ is the standard Laplacian) dates back to the work of H.~Weyl in the early 1900s, while the corresponding study on Riemannian manifolds was initiated by M.~Gaffney in the early 1950s. In particular, the past twenty five years have brought quite a few developments on this topic in the context of the operator $H_{V}|_{C_{c}^{\infty}}$ in the usual $L^2$-space (square integrable functions or sections of a vector bundle) on a Riemannian manifold  $M$, where $C_{c}^{\infty}$ denotes smooth compactly supported functions (or sections), and $H_{V}$ is a Schr\"odinger-type differential expression
\begin{equation}\label{E:general-op}
H_{V}=P+V.
\end{equation}
Here, $P$ is the (magnetic) Laplacian on functions or, more generally, Bochner Laplacian (see section~\ref{SS:s-2-1} below for a description) and $V$ is a locally square integrable ($\lloc^2$) real-valued function (or a self-adjoint section of the appropriate endomorphism bundle).

The literature on the essential self-adjointness of $H_{V}|_{C_{c}^{\infty}}$ showcases two general approaches to the problem---the elliptic equation method and the hyperbolic equation method. For lower semibounded operators, the elliptic equation approach was used by the authors of~\cite{Brezis-79, Simader-78, Wien-58} in the $\mathbb{R}^{n}$ setting and in~\cite{Milatovic-03,sh} in the Riemannian manifold setting, and---for not necessarily lower semibounded operators---by the authors of~\cite{Hinz-Stolz-92,KRB-97,Leinfelder-Simader,RB-70} in the $\mathbb{R}^{n}$ setting and in~\cite{br-c,bms,lm,Oleinik-94,sh-99} in the Riemannian manifold setting. (The authors of~\cite{br-c,lm,Oleinik-94,RB-70,sh-99} assumed $V\in\lloc^{\infty}$.) For lower semibounded $H_{V}|_{C_{c}^{\infty}}$, the hyperbolic equation method was used in~\cite{Povzner-53} in the $\mathbb{R}^{n}$ setting and by the authors of~\cite{Chumak, Chernoff, GK, GP} in the Riemannian manifold setting, and---for not necessarily lower semibounded operators---by the authors of~\cite{Levitan-61,Orochko-76} in the $\mathbb{R}^{n}$ setting and in~\cite{Chernoff-77} in the Riemannian manifold setting. Very recently, in the paper~\cite{HWM-21} the authors explored the links between the essential self-adjointness of Laplace-type operators on Riemannian manifolds (and discrete graphs) and the so-called $L^2$-Liouville property (the phenomenon that the functions belonging to the kernel of the adjoint operator are constant). For additional pointers to the literature, the reader is referred to the books~\cite{ckfs,Guneysu-2016,rs} and papers~\cite{bms,Mel-Roz,Simon-2018}. (In reference to the papers mentioned above, we should add that the authors of~\cite{Chernoff,Levitan-61} studied the operators with  $V\in\lloc^{\infty}$ or more regularity, and in~\cite{Chumak} the case $V=0$ was considered.)

\subsection{Two common types of conditions on $V$}
Looking at the works on the essential self-adjointness of $H_{V}|_{C_{c}^{\infty}}$ (some are referenced in section~\ref{SS:cursory}), the assumptions on $V\in\lloc^2$ are often tailored in a way that the ``negative part" $V^{-}:=\max\{-V,0\}$ satisfies one of the following properties ($d$ is the usual differential, $\Delta$ is the Laplacian on functions on $M$, $1_{G}$ is the characteristic function of a set $G$):
\begin{enumerate}
\item [(H1)] for every compact set $K$, there exist numbers $0\leq a_{K}<1$ and $b_{K}\geq 0$ such that $(1_{K}V^{-}u,u)\leq a_{K}\|d u\|^2+b_{K}\|u\|^2$, for all $u\in C_{c}^{\infty}$;
\item [(H2)] for every compact set $K$, there exist numbers $0\leq a_{K}<1$ and $b_{K}\geq 0$ such that $\|1_{K}V^{-}u\|\leq a_{K}\|\Delta u\|+b_{K}\|u\|$, for all $u\in C_{c}^{\infty}$.
\end{enumerate}
\noindent (Here, $(\cdot,\cdot)$ and $\|\cdot\|$ stand for the inner product and the norm in $L^2$.) When considering~(\ref{E:general-op}), with $P$ being the Bochner Laplacian, one would replace $d$ by the covariant derivative $\nabla$ in (H1) and $V^{-}$ by the norm $|V^{-}|$ in (H2).

Examples of potentials satisfying (H1) and (H2) include those having $V^{-}\in\lloc^p\cap\lloc^2$ with $p$ as in Remark~\ref{R:Guneysu} below, and the examples satisfying (H1) include the potentials with $V^{-}$  (or the norm $|V^{-}|$ in the bundle case) belonging to the local Kato class $\mathcal{K}_{\loc}$ (or, more generally, local contractive Dynkin class $\mathcal{D}_{\loc}$). For a description of $\mathcal{K}_{\loc}$ and $\mathcal{D}_{\loc}$, see section~\ref{S:results-statements} below.

If the lower semiboundedness of $H_{V}|_{C_{c}^{\infty}}$ is not assumed (or is not a consequence of the assumptions on $V$), the authors usually impose a condition of the following type: there exists $\varepsilon\in[0,1]$ such that
\begin{equation}\label{E:epsilon-rb}
\varepsilon(Pu,u)+(Vu,u)\geq -(qu,u),
\end{equation}
for all $u\in C_{c}^{\infty}$, where $P$ and $V$ are as in~(\ref{E:general-op}) and $q(x)=[(\alpha\circ r)(x)]^2$ with $r(x):=d_{g}(x,x_0)$. Here, $\alpha(t)>0$ is a sufficiently regular function tending to $\infty$ not too fast as $t\to\infty$, and the notation $d_{g}(\cdot,x_0)$ indicates the distance from a reference point $x_0$ in the metric $g$ of $M$.

\subsection{A discussion of the first result}
The present paper is concerned with the operator $H_{V}=\nabla^{\dagger}\nabla+V$ on $C_{c}^{\infty}$, where $\nabla$ is a metric covariant derivative (with formal adjoint $\nabla^{\dagger}$) on a Hermitian vector bundle $\vbn$ over a geodesically complete Riemannian manifold $M$ with metric $g$, and  $V\in\lloc^2(\End\vbn)$ is a $\mathcal{D}_{\loc}$-decomposable self-adjoint operator, that is, $V=V_1-V_2$, where $V_j\geq 0$ and $|V_2|\in\mathcal{D}_{\loc}$ (here, $\mathcal{D}_{\loc}$ is the local contractive Dynkin class). As seen in the book~\cite{Guneysu-2016}, the covariant Schr\"odinger operator $\nabla^{\dagger}\nabla+V$ has drawn quite a bit of attention in recent years (in particular, as an object of study in stochastic analysis).

Using a variant of the elliptic equation method, in Theorem~\ref{T:m-1} (in particular, in Corollary~\ref{C:m-1}) of this article we prove that $H_{V}|_{\vcomp}$ is essentially self-adjoint under the following assumptions: $\mathcal{D}_{\loc}$ decomposability of $V$, the fulfillment of~(\ref{E:epsilon-rb}) with $0\leq\varepsilon<1$ and a function $q\geq 1$ (with an additional hypothesis on $q$), and the geodesic completeness of the metric $g_{q}:=q^{-1}g$. (The latter assumption implies that $(M,g)$ is geodesically complete.) We should indicate that a related result, proven in Theorem 2.7 of~\cite{bms}, assumed the fulfillment of (H2) described earlier in this section. It turns out that the condition (H2) is not necessarily satisfied by the potentials $V$ with $|V_2|\in\mathcal{D}_{\loc}$ (or $|V_2|\in\mathcal{K}_{\loc}$). Our Theorem~\ref{T:m-1} is capable of handling $\mathcal{D}_{\loc}$-decomposable potentials thanks to the ``localized self-adjointness" idea from~\cite{Brezis-79} (subsequently implemented in~\cite{KRB-97} in the non-semibounded case), which enables us to keep the proof of the inclusion $\Dom((H_{V}|_{C_{c}^{\infty}})^*)\subset W^{1,2}_{\loc}$ within the realm of quadratic-form inequalities. (Here,  $T^*$ is the operator adjoint of $T$ and $W^{1,2}_{\loc}$ is a local $L^2$-Sobolev space with differential order 1.) With the latter inclusion at our disposal, the remainder of the proof of Theorem~\ref{T:m-1} follows the lines of the refined integration by parts technique as in~\cite{br-c,bms,Oleinik-94, sh-99}.

\subsection{A discussion of the second result}
In Theorem~\ref{T:m-2} we prove that $H_{V}|_{\vcomp}$ has at most one self-adjoint extension in $L^2(\vbn)$ under the following assumptions:  geodesic completeness of $(M,g)$, $\mathcal{D}_{\loc}$ decomposability of $V$, the fulfillment of~(\ref{E:epsilon-rb}) with $\varepsilon=1$, and the requirement that $t\cdot\alpha(t)$ (with $\alpha$ coming from $q(x)=[(\alpha\circ r)(x)]^2$) belongs to the class $\mathscr{V}$ (see section~\ref{S:results-statements} for a description of this class). Examples of $\alpha$ covered by Theorem~\ref{T:m-2} include $\alpha(t)=at+b$ and  $\alpha(t)=(at+b)\ln(t+1)$, with $a>0$ and $b>0$. In the case when  $\alpha$ is constant (see Corollary~\ref{C:m-2-c1} below), we have a lower semibounded operator $H_{V}|_{C_{c}^{\infty}}$ (which, by an abstract fact, is guaranteed to have a self-adjoint extension in $L^2(\vbn)$), and Theorem~\ref{T:m-2} then tells us that $H_{V}|_{\vcomp}$ is essentially self-adjoint in $L^2(\vbn)$.

Corollary~\ref{C:m-2-c1} is a generalization of Theorem~XII.1 of~\cite{Guneysu-2016} (see also Theorem 1.1 of~\cite{GP}), where $|V_2|$ was assumed to belong to the global contractive Dynkin class (which implies lower semiboundedness of $H_{V}|_{C_{c}^{\infty}}$ by Lemma VII.4 in~\cite{{Guneysu-2016}}). We should mention that the assumptions on $V$ in Corollary~\ref{C:m-2-c1} are analogous to those of Theorem 1 in~\cite{GK}, where the authors proved (on a geodesically complete manifold), the essential self-adjointness of Schr\"odinger operators with singular (of the type as in~\cite{Hinz-Stolz-92,Leinfelder-Simader}) magnetic potential. Finally, Corollary~\ref{C:m-2-c2} considers $H_{V}=\Delta+V$ acting on functions, with the same assumptions on $V$ as in Theorem~\ref{T:m-2}, and thanks to Lemma~\ref{L:GT-lemma} below, in this case we get the essential self-adjointness of  $H_{V}|_{\mcomp}$. Corollary~\ref{C:m-2-c2} is a generalization of Corollary 2.3 in~\cite{Milatovic-18}, which assumed $V(x)\geq -[(\alpha\circ r)(x)]^2$ with $\alpha(t)=at+b$, $a>0$, $b>0$.

Unlike Corollary 2.3 in~\cite{Milatovic-18}, which relies on a variant of the elliptic equation method, Theorem~\ref{T:m-2} is obtained using the hyperbolic equation approach, similar to that used in~\cite{Orochko-76} for operators on $\mathbb{R}^{n}$. The approach uses finite propagation speed to establish the representation formula~(\ref{E:formula-or}) for $\cos(A^{1/2}t)f$, where $f\in\vcomp$ and $A$ is a self-adjoint extension
of $H_{V}|_{C_{c}^{\infty}}$ in $L^2(\vbn)$ (if such an extension exists). With this formula at our disposal, we show that, thanks to the assumption $t\cdot\alpha(t)\in \mathscr{V}$,  the function $t\mapsto(\cos(A^{1/2}t)f,f)$, $f\in\vcomp$, is uniquely representable in the class of measures $(E(d\lambda)f,f)$ induced by the spectral resolution $E(\lambda)$ of $A$, which, with the help of Lemma~\ref{L:Vul-lemma} below, leads to the uniqueness of self-adjoint extension
of $H_{V}|_{C_{c}^{\infty}}$ in $L^2(\vbn)$. This variant of the hyperbolic method has a different flavor from that of~\cite{Chernoff-77}, where an ``energy equality" (see equation (7) in~\cite{Chernoff-77}) is established, which is then used (together with the assumption $\alpha(t)=at+b$, $a>0$ and $b>0$) to establish $\dim(\ker((H_{V}|_{C_{c}^{\infty}})^{*}\pm i))=0$. (The author of~\cite{Chernoff-77} considered second-order elliptic operators in divergence form with real-valued coefficients and zero-order term $V$ satisfying $V=V_1-V_2$, $0\leq V_1\in\lloc^2(M)$ and $0\leq V_2\in\lloc^p(M)$, with $p=2$ for $n\leq 3$, $p>2$ for $n=4$ and $p=n/2$ for $n\geq 5$, where $n=\dim M$.) It is not clear if the procedure of~\cite{Chernoff-77} is applicable if we replace $\alpha(t)=at+b$ by more general functions, such as those covered by our Theorem~\ref{T:m-2}.

\subsection{Pointers to some recent papers}

Recently, the authors of~\cite{Pran-Ser-Riz-16} worked in the setting of a not necessarily geodesically complete Riemannian manifold $M$ and, using an ingenious form of the elliptic equation method, obtained, among other things, a far-reaching essential self-adjointess result (see Theorem 1 there) for the operator $\Del_{\omega}+V$, where $\Del_{\omega}$ is the scalar Laplacian with respect to a smooth measure $\omega$ on $M$ and $V\in\lloc^2(M)$. Theorem 1 in~\cite{Pran-Ser-Riz-16} is formulated in terms of the so-called effective potential, which involves the measure $\omega$ and the distance from the metric boundary of $M$. In the geodesically complete case, the requirement imposed on $V$ in Theorem 1 of~\cite{Pran-Ser-Riz-16} reduces to $V(x)\geq -(h(x))^2$, where $h(x)$ is a Lipschitz function on $M$. For the case of $H_{V}=\Delta+V$, where $\Delta$ is the Laplacian on functions on a geodesically complete Riemannian manifold, Corollary~\ref{C:m-2-c2} of our article handles some potentials $V$ that are not covered by Theorem 1 of~\cite{Pran-Ser-Riz-16}. Finally, we mention the recent paper~\cite{FPR-20}, which contains  various essential self-adjointness results for sub-Laplacian-type operators on sub-Riemannian manifolds.

\subsection{Organization of our paper}
Our article consists of six sections and an appendix. In section~\ref{S:res} we give an overview of the notations and state the main results. Section~\ref{S:prelim-t-1} contains some preliminary lemmas. The proofs of Theorem~\ref{T:m-1} and Corollary~\ref{C:m-1} are carried out in section~\ref{S:m-1} and section~\ref{S:C-1} respectively. The proof of Theorem~\ref{T:m-2} is given in section~\ref{S:m-2}. Lastly, in the appendix we recall some auxiliary results.

\section{Results}\label{S:res}
\subsection{Notations}\label{SS:s-2-1} We start with an overview of the notations used in the paper. We will work in the context of a smooth connected Riemannian manifold $M$ (without boundary) equipped with
the volume measure $d\mu$. The symbol $\vbn \to M$ denotes a smooth Hermitian vector bundle over $M$ and $\langle\cdot, \cdot\rangle_{x}$ indicates the corresponding Hermitian structure with the norms $|\cdot|_{x}$ on fibers $\vbn_{x}$. When there is no confusion, the notation $|\cdot|$ indicates the norm of a linear operator $V(x)\colon\vbn_{x}\to\vbn_{x}$. For a fixed $x_0\in M$, we denote
\begin{equation}\label{E:dist-x-0}
r(x):=d_{g}(x_0,x),\qquad x\in M,
\end{equation}
and
\begin{equation}\label{E:b-def-r}
B_{\rho}(x_0):=\{x\in M\colon r(x)<\rho\},
\end{equation}
where $\rho>0$ is a number and $d_{g}(\cdot,\cdot)$ is the distance relative to the metric $g$ on $M$.

We now move over to function spaces. We begin with the symbols $C^{\infty}(\vbn)$ and $\vcomp$, which refer to the smooth sections of $\vbn$ and smooth compactly supported sections of $\vbn$, respectively. The space of square integrable sections of $\vbn$ is denoted by $L^2(\mathcal{V})$ and the usual inner product and the corresponding norm are indicated, respectively, by $(\cdot,\cdot)$ and $\|\cdot\|$. The symbol $W^{k,p}_{\loc}(\vbn)$ stands for the local Sobolev space of sections of $\vbn$, where $k$ and $p$  indicate, respectively, the highest derivative and the corresponding $L^{p}$-space. Specializing to $\vbn=M\times\mathbb{C}$,  we have the following notations: $C^{\infty}(M)$, $\mcomp$, $L^2(M)$, and $W^{k,p}_{\loc}(M)$. Lastly, the symbol $\Omega^1(M)$ indicates smooth 1-forms on $M$.

We now turn to differential operators. We start with a Hermitian covariant derivative on $\vbn$
\begin{equation*}
\nabla\colon C^{\infty}(\vbn)\to C^{\infty}(T^*M\otimes\vbn)
\end{equation*}
and its formal adjoint ${\nabla}^{\dagger}$ with respect to $(\cdot,\cdot)$. The composition 
\begin{equation*}
{\nabla}^{\dagger}\nabla\colon C^{\infty}(\vbn)\to C^{\infty}(\vbn)
\end{equation*}
is often called \emph{Bochner Laplacian}. As a special case, we get the \emph{magnetic Laplacian} on functions $\Delta_{A}:=d_{A}^{\dagger}d_{A}$. Here, the symbol $d_{A}$ denotes the \emph{magnetic differential}
\[
d_{A}u=du+iuA,
\]
where $d\colon C^{\infty}(M)\to \Omega^{1}(M)$ indicates the standard differential, $A\in\Omega^{1}(M)$ stands for a real-valued one-form, and $i$ is the imaginary unit. Specializing to the case $A=0$, we get the (non-negative) Laplace--Beltrami operator $\Delta=d^{\dagger}d$. In this paper, we will be looking at a perturbed Bochner Laplacian
\begin{equation}\label{E:def-H}
H_{V}:={\nabla}^{\dagger}\nabla +V,
\end{equation}
where $V$ is a measurable section $V$ of the bundle $\End \vbn$ and $V(x)\colon\vbn_x\to\vbn_x$ is a self-adjoint operator for almost every $x\in M$. In recent literature (see, for instance, the book~\cite{Guneysu-2016}), the operator $H_{V}$ is called \emph{covariant Schr\"odinger operator}.

\subsection{Statements of results}\label{S:results-statements}

Before stating the first result, we describe a class of minorizing functions for $V$. We say that a function $q\colon M\to \mathbb{R}$ belongs to the class  $\mathscr{M}$ if
\begin{enumerate}
\item [(i)] $q\in\lloc^{\infty}(M)$ and $q\geq 1$;

\item [(ii)] there exists a number $L$ such that $|q^{-1/2}(x)-q^{-1/2}(y)|\leq Ld_{g}(x,y)$, for all $x$, $y\in M$,
where  $d_{g}(\cdot,\cdot)$ is as in~(\ref{E:dist-x-0});

\item [(iii)] $(M, g_{q})$ is geodesically complete, where $g_{q}$ is a metric on $M$ defined as $g_{q}:=q^{-1}g$.
\end{enumerate}

\begin{remark}\label{R:gc-q}
Item (iii) is equivalent to the following condition: for every curve $\gamma$ going to infinity, we have $\int_{\gamma}q^{-1/2}\,ds_{g}=\infty$, where $ds_{g}$ is the arc length element with respect to the metric $g$. Since $q\geq 1$, the property (iii) implies the geodesic completeness of $(M,g)$.

We can also look at $(M,g_{q})$ in (iii) as a manifold which is conformally equivalent to $(M,g)$ with a conformal factor $\tau=q^{-1/2}$ in $g_{q}=\tau^2g=q^{-1}g$, which ``shrinks" the original manifold (because $0<\tau\leq 1$), with the ``shrunken" manifold still being complete with respect to $g_{q}$.
\end{remark}

\begin{theorem}\label{T:m-1} Let $(M, g)$ be a Riemannian manifold and let $\vbn$ be a Hermitian vector bundle over $M$ with a Hermitian covariant derivative $\nabla$. Let $V=V_1-V_2$ with $0\leq V_j\in\lloc^2(\End\vbn)$, $j=1,2$.  Assume that for every $x_0\in M$ there exist numbers $0<\delta\leq 1$, $C\in\mathbb{R}$, and a function $\phi\in\mcomp$ with $0\leq \phi\leq 1$ and $\phi\equiv 1$ on a neighborhood of $x_0$, such that the following properties are satisfied:
\begin{enumerate}
\item [(i)] $(\nabla^{\dagger}\nabla v +\phi V v,v)\geq \delta [(\nabla^{\dagger}\nabla v+\phi V_1v,v)]-C\|v\|^2$,
for all $v\in\vcomp$;
\item [(ii)] $(\nabla^{\dagger}\nabla+\phi V) |_{\vcomp}$ is essentially self-adjoint.
\end{enumerate}
Furthermore, assume that there exists a number $0\leq \varepsilon<1$ and a real-valued function $q\in\mathscr{M}$ such that
\begin{equation}\label{E:hyp-minorant-1}
\varepsilon \|\nabla u\|^2 +(Vu,u)\geq -(qu,u),
\end{equation}
for all $u\in\vcomp$. Then the operator $H_{V}|_{\vcomp}$ is essentially-self-adjoint.
\end{theorem}
\begin{remark} In the context of $M=\mathbb{R}^n$, the quadratic form condition~(\ref{E:hyp-minorant-1}), which is less restrictive than the pointwise condition $V(x)\geq -q(x)$, dates back to the paper~\cite{RB-70}. It turns out that some restrictions (as described in remark~\ref{R:gc-q}) on the function $q$ in~(\ref{E:hyp-minorant-1}) are needed. As an illustration, we recall the example III.1.1 from~\cite{berez-sh-book}:  if $M=\mathbb{R}^n$ and $V=-|x|^{\kappa}$, where $\kappa>2$ and $|x|$ is the Euclidean distance from $x$ to $0$ in $\mathbb{R}^n$, then $-\Delta +V$ is \emph{not} essentially self-adjoint on $C_{c}^{\infty}(\mathbb{R}^n)$. (Here, $\Delta$ is the standard Laplacian on $\mathbb{R}^n$.)
\end{remark}
\begin{remark}\label{R:gc-hamilt}
With regard to the function $q\in\mathscr{M}$ in~(\ref{E:hyp-minorant-1}), it was pointed out in remark 1.2 of~\cite{sh-99} that the condition (iii) from the definition of $\mathscr{M}$ (see also remark~\ref{R:gc-q}) implies the classical completeness of the system~(\ref{E:hamilton-sys}) corresponding to the Hamiltonian $|p|^2-q$. (Here we assume, in addition, that $q\in C^2(M)$.) To quickly illustrate this in the case $M=\mathbb{R}$ with the standard metric, we recall the ``conservation of energy" property along the classical trajectories of the Hamiltonian:
\begin{equation}\nonumber
|p|^2-q(x)=C=\textrm{const}.
\end{equation}
Subsequently, denoting $\dot{x}:=\frac{dx}{dt}$, and referring to the ``conservation of energy" property, we get
\begin{equation}\nonumber
dt=\frac{ds}{\dot{x}}=\frac{ds}{2|p|}=\frac{ds}{2\sqrt{C+q(x)}},
\end{equation}
where $ds$ is the arc-length element.

Thus, the condition (iii) from the definition of $\mathscr{M}$ implies the classical completeness of the Hamiltonian system corresponding to $|p|^2-q$.
\end{remark}

Before stating a corollary containing more specific conditions on $V$ that guarantee the essential self-adjointness of $H_{V}$, we review the concept of (local) Kato and (local) contractive Dynkin classes.  We start with the symbol $p(t,x,y)$, $(t,x,y)\in (0,\infty)\times M\times M$, which indicates the minimal positive heat kernel of $M$ as in Theorem 7.13 in~\cite{Grigoryan-11}. In this paper, $p(t,x,y)$ corresponds to $e^{-t\Delta}$. We can now give a description of the (local) contractive Dynkin class and (local) Kato class. For a Borel function $f\colon M\to\mathbb{C}$ and $t>0$, define
\[
J(t):=\sup_{x\in M}\int_{0}^{t}\int_{M}p(s,x,y)|f(y)|\,d\mu(y)\,ds.
\]
Adopting the terminology of Definition VI.1 in~\cite{Guneysu-2016}, we say that a Borel function $f\colon M\to\mathbb{C}$ is a member of the \emph{contractive Dynkin class} relative to $p(t,x,y)$ and write $f\in \mathcal{D}(M)$ if there exists $t>0$ such that $J(t)<1$. We say that a Borel function $f\colon M\to\mathbb{C}$ is a member of the \emph{Kato class} relative to $p(t,x,y)$ and write $f\in \mathcal{K}(M)$ if $\displaystyle\lim_{t\to 0+}J(t)=0$. The local contractive Dynkin class $\mathcal{D}_{\loc}(M)$ consists of all Borel functions $f\colon M\to \mathbb{C}$ such that for all compact sets $K\subset M$ we have $1_{K}f\in \mathcal{D}(M)$, where $1_{G}$ is the indicator function of a set $G$. The local Kato class $\mathcal{K}_{\loc}(M)$ is defined in the same way as $\mathcal{D}_{\loc}(M)$, with $\mathcal{K}(M)$ in place of  $\mathcal{D}(M)$. We observe that $\mathcal{K}(M)\subset \mathcal{D}(M)$ and $\mathcal{K}_{\loc}(M)\subset \mathcal{D}_{\loc}(M)$.

\begin{remark}\label{R:Guneysu}  Recently, the author of~\cite{Guneysu-2017} showed (see Corollary 2.11 there) that the following property holds on an arbitrary (not necessarily geodesically complete) Riemannian manifold $M$ with $\dim M=n$: for every $1\leq  p<\infty$ such that $p\geq 1$ if $n=1$, and $p>n/2$ if $n\geq 2$, we have $\lloc^{p}(M)\subset \mathcal{K}_{\loc}(M)\subset \mathcal{D}_{\loc}(M)$. To get the analogous global inclusion $L^p(M)\subset \mathcal{K}(M)$, one needs to impose additional requirements on $M$, such as geodesic completeness, positive injectivity radius, and lower semiboundedness of the Ricci curvature; see Theorem 2.9 and Corollary 2.11 in~\cite{Guneysu-2014-ams}.
\end{remark}

We are ready to describe a class of potentials $V$ used in the next corollary (and later in the article). Following the terminology of Definition VII.3 of~\cite{Guneysu-2016}, a section $V$ of $\End\vbn$ will be called \emph{$\mathcal{D}_{\loc}$-decomposable} if $V=V_1-V_2$ with $V_j\in\lloc^2(\End\vbn)$ and $V_j\geq 0$ and $|V_2|\in \mathcal{D}_{\loc}(M)$. If the latter condition is replaced by $|V_2|\in \mathcal{D}(M)$, we say that a section $V$ of $\End\vbn$ is \emph{$\mathcal{D}$-decomposable}. Analogously, we can define the terms \emph{$\mathcal{K}_{\loc}$-decomposable} and \emph{$\mathcal{K}$-decomposable}.

\begin{cor}\label{C:m-1} Let $(M, g)$ be a Riemannian manifold and let $\vbn$ be a Hermitian vector bundle over $M$ with a Hermitian covariant derivative $\nabla$. Assume that $V\in\lloc^{2}(\End \vbn)$ is $\mathcal{D}_{\loc}$-decomposable. Furthermore, assume that there exists a number $0\leq \varepsilon<1$ and a real-valued function $q\in\mathscr{M}$ such that~(\ref{E:hyp-minorant-1}) is satisfied.
Then the operator $H_{V}|_{\vcomp}$ is essentially-self-adjoint.
\end{cor}

Before stating a result concerning the case $\varepsilon=1$ in~(\ref{E:hyp-minorant-1}), we describe the class $\mathscr{V}$ originally introduced by the author of~\cite{Vul-59}. A function $f\colon (0,\infty)\to (0,\infty)$ is said to belong to the class $\mathscr{V}$ if the following conditions are satisfied: $f\in C^1(0,\infty)$; the first derivative $f'$ is increasing on $(0,\infty)$; $\displaystyle\lim_{t\to\infty}\frac{tf'(t)}{f(t)}=\gamma$, where $\gamma$ is finite and $\gamma>1$; there exists $t_0>0$ such that $\int_{t_0}^{\infty}\frac{g(t)}{t^2}\,dt=\infty$, where $g(t)$ is the inverse function corresponding to $f'(t)$. It is easy to check that if $f\in\mathscr{V}$, then $f(t+a)$ and $f(at)$, where $a>0$ is a number, also belong to $\mathscr{V}$. Furthermore, it can be checked that if $m_1(t)=a+bt$ and $m_2(t)=(a+bt)\ln(t+1)$, where $a>0$ and $b>0$ are numbers and $\ln(\cdot)$ is the natural logarithm, then $f_{j}(t):=tm_{j}(t)$, $j=1,2$, belong to $\mathscr{V}$.  We are now ready to state the second result.


\begin{theorem}\label{T:m-2} Let $(M,g)$ be a geodesically complete Riemannian manifold. Let $\vbn$ be a Hermitian vector bundle over $M$ with a Hermitian covariant derivative $\nabla$. Assume that $V\in\lloc^{2}(\End \vbn)$ is $\mathcal{D}_{\loc}$-decomposable. Furthermore, assume that \begin{equation}\label{E:hyp-minorant-2}
(\nabla^{\dagger}\nabla u+Vu,u)\geq -(qu,u),
\end{equation}
for all $u\in\vcomp$, where $q(x)=[(\alpha\circ r)(x)]^2$ with a function $\alpha\colon (0,\infty)\to (0,\infty)$ such that $t\cdot\alpha(t)$ belongs to the class $\mathscr{V}$ and with $r(x)$ as in~(\ref{E:dist-x-0}). Then, one of the following is true: either $H_{V}|_{\vcomp}$ does not have a self-adjoint extension in $L^2(\vbn)$ or the closure $\overline{H_{V}|_{\vcomp}}$ is the only self-adjoint extension of $H_{V}|_{\vcomp}$ in $L^2(\vbn)$ (that is, $H_{V}|_{\vcomp}$ is essentially self-adjoint).
\end{theorem}

If $V\in\lloc^2(\End\vbn)$ is a self-adjoint section and
\begin{equation}\label{E:hyp-minorant-2-addition}
(H_{V}u,u)\geq-c,
\end{equation}
for all $u\in\vcomp$, where $c>0$ is a constant, then (see section B.2 in~\cite{Guneysu-2016}) there exists a self-adjoint extension of $H_{V}|_{\vcomp}$ in $L^2(\vbn)$. If $\alpha(t):=\sqrt{c}+bt$, where $c>0$ is as in~(\ref{E:hyp-minorant-2-addition}) and $b>0$ is a constant, then $t\cdot \alpha(t)$ belongs to the class  $\mathscr{V}$. Moreover, the hypothesis~(\ref{E:hyp-minorant-2-addition}) implies that $H_{V}$ satisfies the condition~(\ref{E:hyp-minorant-2}) with $q(x)=[(\alpha\circ r)(x)]^2$, where $\alpha(t)=\sqrt{c}+bt$ and $r(x)$ is as in~(\ref{E:dist-x-0}). Thus, from Theorem~\ref{T:m-2} we get the following corollary:

\begin{cor}\label{C:m-2-c1} Let $(M,g)$, $\vbn$, and $\nabla$ be as in the hypotheses of Theorem~\ref{T:m-2}.  Assume that $V\in\lloc^{2}(\End \vbn)$ is $\mathcal{D}_{\loc}$-decomposable. Furthermore, assume that $H_{V}$ satisfies the condition~(\ref{E:hyp-minorant-2-addition}). Then, $H_{V}|_{\vcomp}$ is essentially self-adjoint in $L^2(\vbn)$.
\end{cor}

As $(\Delta +V)|_{\mcomp}$, with a real-valued $V\in\lloc^2(M)$, has a self-adjoint extension in $L^2(M)$ (see Lemma~\ref{L:GT-lemma} below), the next corollary is a direct consequence of Theorem~\ref{T:m-2}:

\begin{cor}\label{C:m-2-c2} Let $(M,g)$ be a geodesically complete Riemannian manifold. Assume that $V\in\lloc^{2}(M)$ is a real-valued and $\mathcal{D}_{\loc}$-decomposable function. Assume that~(\ref{E:hyp-minorant-2}) is satisfied with $\Delta$  in place of $\nabla^{\dagger}\nabla$. Then the operator $(\Delta +V)|_{\mcomp}$ is essentially self-adjoint in $L^2(M)$.
\end{cor}

\section{Lemmas used in the proof Theorem~\ref{T:m-1}}\label{S:prelim-t-1}
\subsection{Quadratic forms} Assume that $0\leq Q\in\lloc^2(\End\vbn)$ and consider the sesquilinear forms
\begin{equation}\label{E:forms-h}
h_{\nabla}[u,v]:=\int\langle \nabla u,\nabla v\rangle \,d\mu, \quad h_{Q}[u,v]:=\int \langle Qu,v\rangle\,d\mu,
\end{equation}
with the corresponding quadratic forms $h_{\nabla}[\cdot]$ and $h_{Q}[\cdot]$ and their domains
\[
\Dom(h_{\nabla})=\{u\in L^2(\vbn)\colon h_{\nabla}[u]<\infty\},
\]
\[
\Dom(h_{Q})=\{u\in L^2(\vbn)\colon h_{Q}[u]<\infty\}.
\]
Let $h_1:=h_{\nabla}+h_{Q}$ with the domain $\Dom(h_{1})=\Dom(h_{\nabla})\cap \Dom(h_{Q})$. Note that $\vcomp\subset\Dom(h_{1})$, which makes $h_1$ densely defined. In the subsequent discussion, the symbol $W^{-1,2}(\vbn)$ refers to the anti-dual of $\Dom(h_{\nabla})$.

\subsection{Lemmas}
Having introduced the needed forms, we are ready to state the first lemma, whose proof can be found in Lemma 2.2 of~\cite{Milatovic-03}.
\begin{lemma}\label{L:L-1} If $(M,g)$ is geodesically complete and $0\leq Q\in\lloc^2(\End\vbn)$, then $\vcomp$ is a form core of $h_1$, that is, $\vcomp$ is dense in $\Dom(h_{1})$ with respect to the norm $\|\cdot\|^2_{1}:=h_{\nabla}[\cdot]+h_{Q}[\cdot]+\|\cdot\|^2$, where $\|\cdot\|$ is the norm in $L^2(\vbn)$.
\end{lemma}

\begin{lemma}\label{L:L-2} Assume that $(M,g)$ is geodesically complete.  Let $W=W_1-W_2$, where $0\leq W_j\in\lloc^2(\End\vbn)$, $j=1,2$. Assume that there exist numbers $0\leq \delta\leq 1$ and $C\in\mathbb{R}$ such that
\begin{equation}\label{E:cond-m-1}
(\nabla^{\dagger}\nabla v+Wv,v)\geq\delta [(\nabla^{\dagger}\nabla v+W_1v,v)]-C\|v\|^2,
\end{equation}
for all $v\in\vcomp$.
Let $u\in \Dom(h_{\nabla})\cap \Dom(h_{W_1})$. Then, $\langle W_2u,u\rangle\in L^1(M)$ and
\begin{equation}\label{E:cond-m-2}
\|\nabla u\|^2+(Wu,u)\geq \delta [\|\nabla u\|^2+(W_1u,u)]-C\|u\|^2,
\end{equation}
with $\delta$ and $C$ as in~(\ref{E:cond-m-1}).
\end{lemma}
\begin{proof}
Denoting $h_1:=h_{\nabla}+h_{W_1}$, we can rewrite~(\ref{E:cond-m-1}) as
\begin{equation}\label{E:a-f-1}
(1-\delta)h_1[v]+C\|v\|^2\geq \int_{M}\langle W_2v,v\rangle\,d\mu,
\end{equation}
for all $v\in \vcomp$. Since $u\in \Dom(h_{\nabla})\cap \Dom(h_{W_1})$, using Lemma~\ref{L:L-1} we can find a sequence $v_k\in\vcomp$
approximating $u$ in the norm $\|\cdot\|_1$. In particular, the sequence $\{v_k\}$ has a subsequence, which we also denote by $\{v_k\}$, converging a.e. to $u$. Using~(\ref{E:a-f-1}) with $v=v_k$ and applying Fatou's lemma to the term $(W_2v_k,v_k)$ we see that~(\ref{E:a-f-1}) holds with $u$ in place of $v$. This shows that $\langle W_2u,u\rangle\in L^1(M)$ and that $u$ satisfies the property~(\ref{E:cond-m-2}).
\end{proof}

In the proof of the next lemma (and later in the paper), we will need the concept of the minimal and maximal operator corresponding to~(\ref{E:def-H}) with $V\in\lloc^2(\End\vbn)$. The term \emph{minimal operator} corresponding to $H_{V}$, denoted by $H_{\min}$, refers to the closure of $H_{V}|_{\vcomp}$, while \emph{maximal operator} is defined as $H_{\max}:=(H_{\min})^{*}$, where $T^*$ indicates the adjoint operator corresponding to $T$. It is well known that $\Dom(H_{\max})$ can be described as
\[
\Dom(H_{\max})=\{u\in L^2(\vbn)\colon H_{V}u\in L^2(\vbn)\},
\]
where $H_{V}u$ is interpreted in distributional sense, and $H_{\max}=H_{V}u$ for all $u\in\Dom(H_{\max})$.

\begin{lemma}\label{L:L-3} Assume that $(M,g)$ is geodesically complete.  Let $W=W_1-W_2$ with $0\leq W_j\in\lloc^2(\End\vbn)$, $j=1,2$. Assume that there exist numbers $0<\delta\leq 1$ and $C\in\mathbb{R}$ such that~(\ref{E:cond-m-1}) is satisfied. Additionally, assume that $(\nabla^{\dagger}\nabla+W)|_{\vcomp}$ is essentially self-adjoint. Let $w\in L^2(\vbn)$ and $(\nabla^{\dagger}\nabla+W)w\in W^{-1,2}(\vbn)$. Then, $w\in\Dom(h_{\nabla})\cap \Dom(h_{W_1})$.
\end{lemma}
\begin{proof} With $C$ as in~(\ref{E:cond-m-1}) and $I$ standing for the identity endomorphism of $\vbn$, denote $S:=\nabla^{\dagger}\nabla+W+(C+1)I$, and observe that, by hypothesis, $S|_{\vcomp}$ is essentially self-adjoint and (strictly) positive operator (the latter property holds because of~(\ref{E:cond-m-1})), and, consequently (by Theorem X.26 in~\cite{rs}), we have $\ker(S_{\max})=\ker((S|_{\vcomp})^{*})=\{0\}$. Denote $h_1=h_{\nabla}+h_{W_1}$ and note that by Lemma~\ref{L:L-2}, on the space $\Dom(h_1)=\Dom(h_{\nabla})\cap \Dom(h_{W_1})$ we can define a sesquilinear form
\[
Y[u_1,u_2]:=h_{\nabla}[u_1,u_2]+h_{W_1}[u_1,u_2]-h_{W_2}[u_1,u_2]+(C+1)(u_1,u_2),
\]
where $(\cdot,\cdot)$ is the inner product in $L^2(\vbn)$. Let $(\cdot,\cdot)_{1}$ be the inner product corresponding to the norm $\|\cdot\|^2_{1}:=h_{1}[\cdot]+\|\cdot\|^2$. Looking at the property~(\ref{E:cond-m-2}), as granted by Lemma~\ref{L:L-2}, and adding $(C+1)\|\cdot\|^2$ to both sides, we see that $Y$ is a bounded and coercive form on the Hilbert space $\Dom(h_1)$ with the inner product $(\cdot,\cdot)_{1}$. (The form $Y$ is coercive because in this lemma we assume $\delta>0$.)

Let $w\in L^2(\vbn)$ be as in the hypothesis of this lemma and consider $F:=\nabla^{\dagger}\nabla w+Ww +(C+1)w$, with $C$ as in~(\ref{E:cond-m-2}).  By assumption, we have $F\in W^{-1,2}(\vbn)$, the anti-dual of $\Dom(h_{\nabla})$. As  $\Dom(h_{1})\subset\Dom(h_{\nabla})$, it follows that $F$ is a bounded linear functional the Hilbert space $\Dom(h_1)$ equipped with the inner product $(\cdot,\cdot)_{1}$. Therefore, by Lax--Milgram Theorem (see Theorem 5.8 in~\cite{GT-book}) there exists a unique element $u\in \Dom(h_{1})$ such that $Y(u, s)=F(s)$, for all $s\in\Dom(h_1)$. In particular, for all $s\in\vcomp\subset\Dom(h_1)$, the equality $Y(u, s)=F(s)$ can be written as
\[
(\nabla u,\nabla s)+(W_1u,s)-(W_2u,s)+(C+1)(u,s)=(\nabla^{\dagger}\nabla w+Ww +(C+1)w,s)_{d},
\]
where $(T,s)_{d}$ stands for the action of a distributional section $T$ on a section $s\in\vcomp$.
Using the integration by parts (see Lemma 8.8 in~\cite{bms}) in the term $(\nabla u,\nabla s)$ and remembering that $W_ju\in\lloc^1(\vbn)$, $j=1,2$, the last equation can be rewritten as
\[
(\nabla^{\dagger}\nabla u+Wu +(C+1)u,s)_{d}=(\nabla^{\dagger}\nabla w+Ww +(C+1)w,s)_{d},
\]
that is, $(\nabla^{\dagger}\nabla +W +(C+1)I)(u-w)=0$ in distributional sense, which shows that $(u-w)\in\ker (S_{\max})$. Remembering the property $\ker(S_{\max})=\{0\}$, we see that $w=u$, which means that $w\in\Dom(h_1)=\Dom(h_{\nabla})\cap \Dom(h_{W_1})$.
\end{proof}

In the subsequent discussion we will use the following formula with $f\in C^{\infty}(M)$ and $u\in L^2(\vbn)$:
\begin{equation}\label{E:prod-2}
\bo(fu)=f\bo u-2\nabla_{(df)^{\sharp}}u +u\Delta f.
\end{equation}
Here, $\omega^{\sharp}$ stands for the vector field corresponding to the 1-form $\omega$ with respect to the metric $g$.

\begin{lemma}\label{L:L-4} Assume that $(M,g)$ is geodesically complete. Let $V=V_1-V_2$ with $0\leq V_j\in\lloc^2(\End\vbn)$, $j=1,2$.  Assume that $V$ satisfies the assumptions (i) and (ii) of Theorem~\ref{T:m-1}. Let $u\in\Dom (H_{\max})$,  where $H_{V}$ is as in~(\ref{E:def-H}). Then, $u\in W^{1,2}_{\loc}(\vbn)$ and $\langle V_1u,u\rangle\in\lloc^1(M)$.
\end{lemma}

\begin{proof} Let $x_0\in M$ be arbitrary, let $\phi\in\mcomp$ be as in the assumptions (i) and (ii) of Theorem~\ref{T:m-1}, and let $\chi\in\mcomp$ be a function satisfying $0\leq\chi\leq 1$, $\chi(x_0)=1$, and $\supp\chi\subset\{x\in M:\phi(x)=1\}$, where $\supp\chi$ denotes the support of $\chi$. If $u\in\Dom(H_{\max})$, then
\begin{align}\label{E:temp-ch-1}
&\nabla^{\dagger}\nabla(\chi u)+\phi V (\chi u)=\nabla^{\dagger}\nabla(\chi u)+V (\chi u)\nonumber\\
&=\chi H_{V}u-2\nabla_{(d\chi)^{\sharp}}u +u\Delta \chi,
\end{align}
where in the first equality we used the property of $\supp\chi$ and in the second equality we used~(\ref{E:prod-2}) and the definition of $H_{V}$.
As $u\in L^2(\vbn)$ and $H_{V}u\in L^2(\vbn)$ we infer that
\begin{equation}\label{E:w-1-2-t}
(\nabla^{\dagger}\nabla(\chi u)+\phi V (\chi u))\in W^{-1,2}(\vbn).
\end{equation}
Since $V$ satisfies the assumptions (i) and (ii) of Theorem~\ref{T:m-1}, we see that $W:=\phi V$ satisfies the hypotheses of Lemma~\ref{L:L-3}. With the properties $\chi u\in L^2(\vbn)$ and~(\ref{E:w-1-2-t}) at our disposal, from Lemma~\ref{L:L-3} we get $\chi u\in\Dom(h_{\nabla})\cap \Dom(h_{\phi V_1})$. Since $x_0\in M$ is arbitrary, the first inclusion means that $u\in W^{1,2}_{\loc}(\vbn)$ and the second one says $\langle \phi V_1\chi u,\chi u\rangle\in L^1(M)$, that is, $\langle V_1 u, u\rangle\in\lloc^1(M)$.
\end{proof}



We now state the last lemma of this section.

\begin{lemma}\label{L:L-6} Assume that $(M,g)$ is geodesically complete. Let $V=V_1-V_2$ with $0\leq V_j\in\lloc^2(\End\vbn)$, $j=1,2$.  Assume that $V$ satisfies the assumptions (i) and (ii) of Theorem~\ref{T:m-1}.  Furthermore, assume that there exists a number $0\leq \varepsilon<1$ and a function $q\in\lloc^{\infty}(M)$ with $q\geq 0$ such that
\begin{equation}\label{E:hyp-minorant-posit}
\varepsilon \|\nabla v\|^2 +(Vv,v)\geq -(qv,v),
\end{equation}
for all $v\in\vcomp$. Let $H_{V}$ be as in~(\ref{E:def-H}). Then, for every Lipschitz compactly supported function $\psi\colon M\to\mathbb{R}$ and every $u\in\Dom(H_{\max})$, we have
\begin{equation}\label{E:cond-m-3-posit}
\varepsilon(\nabla^{\dagger}\nabla (\psi u),\psi u)+(V\psi u,\psi u) \geq -(q\psi u, \psi u).
\end{equation}
\end{lemma}
\begin{proof}
Since $u\in\Dom(H_{\max})$, Lemma~\ref{L:L-4} tells us that $u\in W^{1,2}_{\loc}(\vbn)$ and $\langle V_1u,u\rangle\in\lloc^1(M)$. Therefore, as $\psi$ is a compactly supported Lipschitz function, we have $\psi u\in \Dom(h_{\nabla})\cap \Dom(h_{V_1})$. Assuming $0<\varepsilon<1$ for the moment, we  rewrite~(\ref{E:hyp-minorant-posit}) as follows:
\begin{equation}\nonumber
(\nabla^{\dagger}\nabla v,v)+\varepsilon^{-1}((V+q)v,v) \geq 0.
\end{equation}
Remembering that $\psi u\in \Dom(h_{\nabla})\cap \Dom(h_{V_1})$ and taking into account $0\leq q\in\lloc^{\infty}(M)$, we obtain $\psi u\in \Dom(h_{\nabla})\cap\Dom(h_{W_1})$, where $W_1=\varepsilon^{-1}(V_1+q)\geq 0$. Therefore, we can use Lemma~\ref{L:L-2} with $\delta=C=0$, $W_1=\varepsilon^{-1}(V_1+q)$ and $W_2=\varepsilon^{-1}V_2$ to obtain $\langle W_2\psi u,\psi u\rangle\in L^1(M)$ and
\begin{equation}\nonumber
(\nabla^{\dagger}\nabla (\psi u),\psi u)+\varepsilon^{-1}((V+q)(\psi u),\psi u) \geq 0,
\end{equation}
which leads to~(\ref{E:cond-m-3-posit}). For $\varepsilon=0$, the inequality~(\ref{E:hyp-minorant-posit}) can rewritten as $((V+q)v,v)\geq 0$ for all $v\in\vcomp$. To get~(\ref{E:cond-m-3-posit}) for $\varepsilon=0$, we can use an argument based on Friedrichs mollifiers (see the proof of Lemma 2.2 in~\cite{Milatovic-03} for details).
\end{proof}

\section{Proof of Theorem~\ref{T:m-1}} \label{S:m-1}
We first recall two ``product rule" formulas, which we will use for a Lipschitz compactly supported function $\psi$ and $u\in W^{1,2}_{\loc}(\vbn)$:
\begin{equation}\label{E:prod-1}
\nabla(\psi u)=d\psi\otimes u+\psi\nabla u,\qquad
\end{equation}
and
\begin{equation}\label{E:prod-1-d}
\nabla^{\dagger}(\psi\nabla u)= \psi \nabla^{\dagger}\nabla u-\nabla_{(d\psi)^{\sharp}}u,
\end{equation}
where $(d\psi)^{\sharp}$ is understood as in~(\ref{E:prod-2}).

Next we recall (see, for instance, Lemma 8.9 in~\cite{bms}) that if $(M,b)$ is a geodesically complete Riemannian manifold with metric $b$, then there exists a
sequence of compactly supported Lipschitz functions $\{\phi_k\}$ such that
\begin{enumerate}
\item  [(L1)] $0\leq\phi_k\leq 1$ and $|d\phi_k|_{b}\leq\frac{1}{k}$, where $|d\phi_k|_{b}$
indicates the length of the cotangent vector $d\phi_k$ corresponding to the metric $b$;
\item [(L2)] $\displaystyle\lim_{k\to\infty}\phi_k(x)=1$, for all $x\in M$.
\end{enumerate}
With the preparations carried out in section~\ref{S:prelim-t-1}, the proof of Theorem~\ref{T:m-1} is practically the same as that of Theorem 2.7 of~\cite{bms}. Nevertheless, as our preparations are different from those in~\cite{bms} and as we will refer to specific items from section~\ref{S:prelim-t-1}, we will show the details of the proof. From now on in this section, we assume that the hypotheses of Theorem~\ref{T:m-1} are satisfied. Without stating it explicitly each time, $q$, $V$, $\varepsilon$, and $H_{V}$ are as in Theorem~\ref{T:m-1}.
\begin{lemma}\label{L:axil} Let $q$ and $L$ be as in Theorem~\ref{T:m-1}. Then, for all $u\in\Dom(H_{\max})$ we have $(q^{-1/2}\nabla u) \in L^2(T^*M \otimes\vbn)$ and
\begin{equation}\label{E:toprove-bms}
    \|q^{-1/2} \nabla u\|^{2} \leq \frac{2}{1-\varepsilon}\left(\left(1+\frac{2L^2(1+\varepsilon)^2}{1-\varepsilon}\right)\|u\|^{2}
                                    +\|u\|\|H_{V}u\|\right).
\end{equation}
\end{lemma}
\begin{proof}
Let $\psi$ be a Lipschitz compactly supported function satisfying the inequality $0\leq\psi\leq q^{-1/2}\le 1$ and let
$K:=\underset{x\in M}\esssup |d\psi(x)|_{g}$, where $|d\psi|_{g}$
is the length of the cotangent vector $d\psi$ corresponding to the metric $g$. We will first show that for all $u\in\Dom(H_{\max})$, the inequality~(\ref{E:toprove-bms}) holds with $K$ and $\psi$ in place of $L$ and $q^{-1/2}$ respectively. Since $u\in \Dom(H_{\max})$, Lemma~\ref{L:L-4} tells us that $u\in W^{1,2}_{\loc}(\vbn)$. Therefore,
\begin{align}\label{E:bms-1}
\|\psi \nabla u\|^2 =(\psi \nabla u, \psi \nabla u)&= (\nabla^{\dagger}(\psi^2\nabla u),u)=(\psi^{2}\nabla^{\dagger}\nabla u,u) - 2(\psi \nabla_{(d\psi)^{\sharp}}u,u)\nonumber\\
&=\RE(\psi^{2}\nabla^{\dagger}\nabla u,u) -2\RE(\psi \nabla_{(d\psi)^{\sharp}}u,u)\nonumber\\
&\leq \RE(\psi^{2}\nabla^{\dagger}\nabla u,u) + 2K\|\psi \nabla u\|\|u\|,
\end{align}
where in the second equality we used integration by parts (see Lemma 8.8 of~\cite{bms}), the application of which is allowed in our case since $\psi^2 u\in W^{1,2}_{\comp}(\vbn)$, in the third equality we used~(\ref{E:prod-1-d}) and the property $d(\psi^2)=2\psi d\psi$, and, lastly, we used the definition of $K$. After rearranging~(\ref{E:bms-1}) we get
\begin{align}\label{E:bms-2}
&\RE(\psi^{2}\nabla^{\dagger}\nabla u,u)\geq (\psi \nabla u, \psi \nabla u)-2K\|\psi \nabla u\|\|u\|\nonumber\\
=(\nabla (\psi u), \nabla (\psi u))&-(d\psi\otimes u,\psi \nabla u)-(\psi \nabla u,d\psi\otimes u)-2K\|\psi \nabla u\|\|u\|\nonumber\\
&\geq (\nabla (\psi u), \nabla (\psi u))-4K\|\psi \nabla u\|\|u\|,
\end{align}
where in the equality we used~(\ref{E:prod-1}) and in the second inequality we used the definition of $K$.

Multiplying~(\ref{E:bms-1}) by $1-\varepsilon$ and doing further estimates (the steps are explained below) we obtain
\begin{align}\label{E:bms-3}
&(1-\varepsilon)\|\psi \nabla u\|^2\leq (1-\varepsilon)\RE(\psi^{2}\nabla^{\dagger}\nabla u,u) + 2K(1-\varepsilon)\|\psi \nabla u\|\|u\|\nonumber\\
&=\RE(\psi^2H_{V}u,u)-\varepsilon\RE(\psi^2\nabla^{\dagger}\nabla u,u) -(\psi^2Vu,u)+2K(1-\varepsilon)\|\psi \nabla u\|\|u\|\nonumber\\
&\leq \|H_{V}u\|\|u\| - \varepsilon(\nabla(\psi u),\nabla(\psi u))-(V(\psi u),\psi u)+2K(1+\varepsilon)\|\psi \nabla u\|\|u\|\nonumber\\
&\leq \|H_{V}u\|\|u\| + (q\psi u,\psi u)+2K(1+\varepsilon)\|\psi \nabla u\|\|u\|\nonumber\\
&\leq \|H_{V}u\|\|u\|+ \|u\|^2 + 2K(1+\varepsilon)\|\psi \nabla u\|\|u\|,
\end{align}
where in the equality we used the definition of $H_{V}$ and in the second estimate we applied Cauchy-Schwarz inequality to $(\psi^2H_{V}u,u)$. Additionally, in the second estimate, the term $-\varepsilon(\nabla(\psi u),\nabla(\psi u))$ and the change from $2K(1-\varepsilon)$ to $2K(1+\varepsilon)$  occur as a result of using~(\ref{E:bms-2}) multiplied by $-\varepsilon$. Furthermore, in the third inequality we used Lemma~\ref{L:L-6}, and in the fourth inequality we used the hypothesis $0\leq\psi\leq q^{-1/2}$. With the help of $2ab\leq\tau a^{2}+\frac{b^{2}}{\tau}$, $(a,b\in\RR)$ with $\tau=\frac{1-\varepsilon}{2(1+\varepsilon)^2}$, we can estimate the last term on the right hand side of~(\ref{E:bms-3}) to get
\[
2(1+\varepsilon)K\|\psi \nabla u\|\|u\|\leq  \frac{1-\varepsilon}2\|\psi \nabla u\|^{2}+
        2K^{2}\frac{(1+\varepsilon)^2}{1-\varepsilon}\|u\|^{2},
\]
which when combined with~(\ref{E:bms-3}) leads to
\begin{equation}\label{E:temp-toprove-1}
\frac{1-\varepsilon}2\|\psi \nabla u\|^{2}\leq \Big(\, 1+\frac{2K^2(1+\varepsilon)^2}{1-\varepsilon}\, \Big)\|u\|^{2}
                +\|u\|\|H_{V}u\|.
\end{equation}

As $(M,g)$ is geodesically complete (see Remark~\ref{R:gc-q}), there exists a sequence $\phi_k$ of Lipschitz compactly supported functions satisfying (L1)--(L2) above with $b=g$. Letting
$\psi_{k}:=\phi_k\cdot q^{-1/2}$, observe that $0\leq\psi_{k}\leq q^{-1/2}\leq 1$ and
\[
|d\psi_{k}|_{g}\le|d\phi_k|_{g}\cdot q^{-1/2}+\phi_k|dq^{-1/2}|_{g}.
\]
Therefore, $|d\psi_{k}|_{g}\leq\frac{1}{k}+L$, where $L$ is  as in item~(ii) of the definition of $\mathscr{M}$.  Noting that
$\psi_{k}(x)\to q^{-1/2}(x)$ as $k\to\infty$ and using the dominated convergence theorem in~(\ref{E:temp-toprove-1}) with $\psi=\psi_k$ we obtain, after dividing both sides by $(1-\varepsilon)/2$, the inequality~(\ref{E:toprove-bms}).
\end{proof}

\medskip

\noindent\textbf{Continuation of the proof of Theorem~\ref{T:m-1}} To prove that $H_{V}|_{\vcomp}$ is essentially self-adjoint, it is enough to show that $H_{\max}$ is a symmetric operator. Let $u,v\in \Dom(H_{V,\max})$ and let $\psi\geq0$ be a Lipschitz compactly supported function.  As seen in Lemma~\ref{L:L-4}, the sections $u$ and $v$ belong to $W^{1,2}_{\loc}(\vbn)$; hence, we can use integration by parts (see Lemma 8.8 in~\cite{bms}) and~(\ref{E:prod-1}) to get
\begin{equation}\nonumber
    (\psi u,\nabla^{\dagger}\nabla v)=(\psi \nabla u, \nabla v) + (d\psi\otimes u,\nabla v),
\end{equation}
and, similarly,
\begin{equation}\nonumber
    (\nabla^{\dagger}\nabla u, \psi v)=(\nabla u, \psi \nabla v) + (\nabla u,d\psi\otimes v),
\end{equation}
where $(\cdot,\cdot)$ on the left-hand sides of both equations stands for the anti-duality between $W^{-1,2}_{\loc}(\vbn)$ and $W^{1,2}_{\comp}(\vbn)$. Keeping in mind (see Lemma~\ref{L:axil}) that $q^{-1/2}\nabla u$ and $q^{-1/2}\nabla v$ belong to
$L^2(T^*M \otimes\vbn)$, from the last two equations we get
\begin{align}\label{E:pravlev}
&|(\psi u,H_{V}v)-(H_{V}u,\psi v)|\leq  |(d\psi\otimes u,\nabla v)|+|(\nabla u,d\psi\otimes v)|\nonumber\\
&\leq \underset{x\in M}\esssup \left(|d\psi|_{g}q^{1/2}\right)\cdot\left(\|u\|\|q^{-1/2}\nabla v\|+\|v\|\|q^{-1/2}\nabla u\|\right),
\end{align}
where $|d\psi|_{g}$ indicates the length of the cotangent vector $d\psi$ in metric $g$. Recalling that $(M,g_{q})$ is geodesically complete, where $g_{q}=q^{-1}g$ is as in (iii) of the definition of $\mathscr{M}$,  we let $\phi_k$ be a sequence of Lipschitz compactly supported functions on $M$ satisfying (L1)--(L2) above with $b=g_{q}$.  From the property
$|d\phi_k|_{g_{q}}=q^{1/2}|d\phi_k|_{g}$ and (L2) with $b=g_{q}$ we get
\[
\underset{x\in M}\esssup(|d\phi_k|_{g}q^{1/2}(x))\leq\frac{1}{k},
\]
which, when combined with~(\ref{E:pravlev}) with $\psi=\phi_k$, leads to
\[
|(\phi_k u,H_{V}v)-(H_{V}u,\phi_k v)|\leq \frac{1}{k}\left(\|u\|\|q^{-1/2}\nabla v\|+\|v\|\|q^{-1/2}\nabla u\|\right).
\]
Letting $k\to\infty$ in the last inequality and using the dominated convergence theorem on the left hand side, we obtain $(H_{V}u,v)=(u,H_{V}v)$, for all $u, v\in \Dom(H_{\max})$, which shows that $H_{\max}$ is symmetric. \hfill$\square$

\section{Proof of Corollary~\ref{C:m-1}}\label{S:C-1}
It suffices to show that the hypotheses (i) and (ii) of Theorem~\ref{T:m-1} are satisfied. For every $x_0\in M$, we can find a function $\phi\in\mcomp$ with $0\leq \phi\leq 1$ and $\phi\equiv 1$ on a neighborhood of $x_0$. If $V$ is $\mathcal{D}_{\loc}$-decomposable with decomposition $V=V_1-V_2$, $0\leq V_j\in\lloc^2(\End\vbn)$,  then from the definition of $\mathcal{D}_{\loc}(M)$ we get $\phi |V_2|\in\mathcal{D}(M)$. Thus, we can use Theorem XII.1 in~\cite{Guneysu-2016} to conclude that
$(\nabla^{\dagger}\nabla+\phi V) |_{\vcomp}$ is essentially self-adjoint, which means that the condition (ii) of Theorem~\ref{T:m-1} is satisfied. Since $\phi |V_2|\in\mathcal{D}(M)$, by Lemma VII.4 in~\cite{Guneysu-2016} the form $h_{\phi V_2}$, defined as in~(\ref{E:forms-h}) with $Q=\phi V_2$, is $h_{\nabla}$-bounded with bound $<1$, that is, there exist numbers $0\leq a<1$ and $C\geq 0$ such that
\begin{equation}\label{E:temp-ch-3}
h_{\phi V_2}[v,v]\leq a h_{\nabla}[v,v]+C\|v\|^2,
\end{equation}
for all $v\in\Dom(h_{\nabla})$. Therefore, keeping in mind that $\phi V_1\geq 0$, for all $v\in\vcomp$ we have
\begin{align}\nonumber
&(\nabla^{\dagger}\nabla v +\phi V v,v)=(\nabla^{\dagger}\nabla v +\phi V_1 v,v)-h_{\phi V_2}[v,v]\nonumber\\
&\geq (\nabla^{\dagger}\nabla v +\phi V_1 v,v)-a h_{\nabla}[v,v]-C\|v\|^2\nonumber\\
&\geq(\nabla^{\dagger}\nabla v +\phi V_1 v,v)-ah_{\nabla}[v,v]-a(\phi V_1v,v) -C\|v\|^2\nonumber\\
&=(1-a)[(\nabla^{\dagger}\nabla v +\phi V_1 v,v)]-C\|v\|^2.\nonumber
\end{align}
Thus, the condition (i) of  Theorem~\ref{T:m-1} is satisfied with $\delta=1-a$.
$\hfill\square$

\section{Proof of Theorem~\ref{T:m-2}}\label{S:m-2}

We first recall some abstract facts. Let $\mathscr{H}$ be a Hilbert space. Let $A$ be a self-adjoint operator (with domain $\Dom(A)$) in $\mathscr{H}$, let $f\in \Dom(A)$, and let $T>0$ be a number. By a solution $U(t)$ of the equation
\begin{equation}\label{E:2-1-or}
U''+AU=0,\quad U(0)=f,\,\, U'(0)=0,
\end{equation}
we mean a $\Dom(A)$-valued function $U(t)$, $t\in[0,T]$, such that $U(t)$ is twice differentiable on $[0,T]$, the equality $U''(t)+AU(t)=0$ holds for all $t\in [0,T]$, and the initial conditions $U(0)=f$ and $U'(0)=0$ are satisfied. For the following abstract fact, see Lemma 3 in~\cite{Orochko-81}:

\begin{lemma}\label{L:2-1} Let $A$ be a self-adjoint (not necessarily lower semibounded) operator in a Hilbert space $\mathscr{H}$ and let $T>0$ be a number. Assume that $U(t)$, $t\in [0,T]$, satisfies~(\ref{E:2-1-or}) with some $f\in\Dom(A)$. Then, for all $t\in [0,T]$ we have $f\in\Dom(\cos(A^{1/2}t))$ and $U(t)=\cos(A^{1/2}t)f$.
\end{lemma}

In the next two lemmas we consider consider $H_{W}:=\nabla^{\dagger}\nabla +W$ with $W\in\lloc^2(\End\vbn)$ being $\mathcal{D}$-decomposable, where the latter notion was defined before the statement of Theorem~\ref{T:m-1}. Under this hypothesis on $W$ and assuming that $(M,g)$ is geodesically complete, the operator $H_{W}|_{\vcomp}$ is lower semibounded and essentially self-adjoint (see Theorem 1.1 in~\cite{GP} or Theorem  XII.1 in~\cite{Guneysu-2016}). The symbol $\overline{H_{W}}$ indicates the self-adjoint closure of $H_{W}|_{\vcomp}$, where we dropped $|_{\vcomp}$ for simplicity. The following property was proven in Lemma XII.4 (see also equation XII.13) of~\cite{Guneysu-2016}:

\begin{lemma}\label{L:2-2} Assume that $(M,g)$ is geodesically complete. Assume that $W\in\lloc^2(\End\vbn)$ is $\mathcal{D}$-decomposable. Let $f\in \Dom(\overline{H_{W}})$ be a compactly supported section such that $\supp f\subset B_{R}(x_0)$, with $B_{R}(x_0)$ as in~(\ref{E:b-def-r}). Then, $u(t,\cdot):= [\cos ((\overline{H_{W}})^{1/2}t)f](\cdot)$ has the following property: $\supp u(t,\cdot)\subset B_{R+t}(x_0)$, for all $t\geq 0$.
\end{lemma}

The following Lemma is analogous to Lemma 4.2 of~\cite{Chernoff-77}, which was proven for Schr\"odinger operators acting on scalar-valued functions on a compact Riemannian manifold and assumptions on the potential different from those in this article.

\begin{lemma}\label{L:2-2-5} Assume that $(M,g)$ is geodesically complete. Assume that $W\in\lloc^2(\End\vbn)$ is $\mathcal{D}$-decomposable. Let $\rho_1>0$ and let $f\in \Dom(\overline{H_{W}})$ be a compactly supported section such that $\supp f\subset B_{\rho_{1}}(x_0)$. Then, for all $\rho>\rho_1$ there exists a sequence $f_j\in\vcomp$ with $\supp f_j\subset B_{\rho}(x_0)$ such that
\begin{enumerate}
\item [(i)] $\displaystyle \lim_{j\to\infty} f_j=f$;
 \item [(ii)] $\displaystyle\lim_{j\to\infty} H_{W}f_j=\overline{H_{W}}f$,
\end{enumerate}
where the limits are taken with respect to the norm of $L^2(\vbn)$.
\end{lemma}
\begin{proof} Let $\rho_1>0$ and let $f\in \Dom(\overline{H_{W}})$ be a compactly supported section such that $\supp f\subset B_{{\rho_1}}(x_0)$.
By the definition of the closure $\overline{H_{W}}$, there exists a sequence $v_k\in \vcomp$ such that $v_k\to f$ and $H_{W}v_k\to \overline{H_{W}}f$ in $L^2(\vbn)$. Let $\chi\in\mcomp$ be a function satisfying the following properties: $0\leq \chi\leq 1$; $\chi(x)\equiv 1$ in a neighborhood of $\supp f$; and $\supp \chi\subset B_{\rho}(x_0)$, where $\rho>\rho_1$. Define $u_k:=\chi v_k$ and observe that $u_k\in\vcomp$ with $\supp u_k\subset B_{\rho}(x_0)$. Furthermore, note that $u_k\to \chi f=f$ in $L^2(\vbn)$. Similarly as in~(\ref{E:temp-ch-1}), we have
\begin{align}\label{E:temp-ch-2}
&H_{W}u_k=H_{W}(\chi v_k)=\chi H_{W}v_k-2\nabla_{(d\chi)^{\sharp}}v_k +v_k\Delta \chi,
\end{align}
from which we see that the first and the third term on the right hand side converge (in the norm of $L^2(\vbn)$) to $\chi \overline{H_{W}}f=\overline{H_{W}}f$ and $f\Delta \chi=0$, respectively, where we have used the definition of $v_k$ and the fact that $\chi(x)\equiv 1$ in a neighborhood of $\supp f$. Next, we show that there is a subsequence of $v_k$, again denoted by $v_k$, such that $\nabla_{(d\chi)^{\sharp}}v_k\to 0$ weakly. For the latter property, it suffices to show that the sequence $\nabla_{(d\chi)^{\sharp}}v_k$ is (norm) bounded in $L^2(\vbn)$, after taking into account that $\nabla_{(d\chi)^{\sharp}}v_k\to 0$ in distributional sense:
\begin{align}\nonumber
&(\nabla_{(d\chi)^{\sharp}}v_k,s)=(\nabla v_k,d\chi\otimes s)=(v_k, (\Delta \chi)s)-(v_k,\nabla_{(d\chi)^{\sharp}}s)\nonumber\\
&\to (f,(\Delta \chi)s)-(f,\nabla_{(d\chi)^{\sharp}}s)=0,\nonumber
\end{align}
as $k\to 0$, for all $s\in\vcomp$, where in the second equality we used integration by parts and in the last equality we used the fact that $\chi(x)\equiv 1$ in a neighborhood of $\supp f$. Finally, thanks to
\[
\|\nabla_{(d\chi)^{\sharp}}v_k\|\leq c_1\|\nabla v_k\|,
\]
where $c_1$ is a constant, it is enough to show that $\nabla v_k$ is a (norm) bounded sequence in $L^2(T^*M\otimes\vbn)$.
Using a decomposition $W=W_1-W_2$, $0\leq W_{j}\in\lloc^2(\End \vbn)$ with $|W_2|\in\mathcal{D}(M)$, and proceeding as in~(\ref{E:temp-ch-3}), we see that there exist numbers $0\leq a<1$ and $C\geq 0$ such that
\begin{align}\nonumber
&\|\nabla v_k\|^2=(\nabla^{\dagger}\nabla v_k,v_k)=(H_{W}v_k,v_k)-(W_1v_k,v_k)+(W_2v_k,v_k)\nonumber\\
&\leq (H_{W}v_k,v_k)+a\|\nabla v_k\|^2+C\|v_k\|^2,\nonumber
\end{align}
which leads to
\[
\|\nabla v_k\|^2\leq (1-a)^{-1}[(H_{W}v_k,v_k)+C\|v_k\|^2].
\]
The latter inequality tells us that $\|\nabla v_k\|$ is bounded (because the sequences $H_{W}v_k$ and $v_k$ converge in $L^2(\vbn)$). Thus, we have shown that there is a subsequence of $u_k\in\vcomp$, which we again denote by $u_k$, such that $H_{W}u_k\to \overline{H_{W}}f$ weakly in $L^2(\vbn)$. Now we can start with $u_k$ and use Mazur's Lemma (see Lemma~\ref{L:Maz-L} below) to construct a sequence $f_j\in\vcomp$ with $\supp f_j\subset B_{\rho}(x_0)$ satisfying the properties (i) and (ii) of this lemma.
\end{proof}

From hereon we assume that the hypotheses of Theorem~\ref{T:m-2} are satisfied. To make our notations simpler, we drop $x_0$ from the symbol $B_{\rho}(x_0)$. For $\rho>0$ define $V_{\rho}(x):=V(x)$ if $x\in B_{\rho}$ and $V_{\rho}(x):=0$ if $x\notin B_{\rho}$. Note that $V_{\rho}$ is $\mathcal{D}$-decomposable. Denoting
\begin{equation}\label{E:2-2a-or}
H_{V_{\rho}}:= \nabla^{\dagger}\nabla+V_{\rho}
\end{equation}
and looking at the comments before Lemma~\ref{L:2-2}, we see that $H_{V_{\rho}}|_{\vcomp}$ is a lower semibounded and essentially self-adjoint operator. Thus, Lemma~\ref{L:2-2} and Lemma~\ref{L:2-2-5} are applicable to $H_{V_{\rho}}$.

Let $H_{V}$ and $H_{V_{\rho}}$ be as in~(\ref{E:def-H}) and~(\ref{E:2-2a-or}) respectively. As $H_{V_{\rho}}|_{\vcomp}$ is lower semibounded and $\vcompr\subset\vcomp$, it follows that $H_{V}|_{\vcompr}=H_{V_{\rho}}|_{\vcompr}$ is lower semibounded. Let $F_{\rho}$, $\rho>0$, denote the Friedrichs extension of $H_{V}|_{\vcompr}$ in the space $L^2(\vbn|_{B_{\rho}})$.

\begin{lemma} \label{L:2-3} Let $H_{V}$ be as in~(\ref{E:def-H}). Assume that the hypotheses of Theorem~\ref{T:m-2} are satisfied. Assume that $H_{V}|_{\vcomp}$ has a self-adjoint extension $A$ in $L^2(\vbn)$.  Let $T_1>T_0>0$ and $R>0$ be arbitrary real numbers and let $\rho:=R+T_1$. Then, for all $t\in [0,T_0]$ and for all $f\in\vcomp$ such that $\supp f\subset B_{R}$ we have
\begin{equation}\label{E:formula-or}
[\cos(A^{1/2}t)f](x)=[\cos(({F_{\rho}})^{1/2}t)f](x),
\end{equation}
where $F_{\rho}$ is the Friedrichs extension of $H_{V}|_{\vcompr}$ in $L^2(\vbn|_{B_{\rho}})$.
\end{lemma}
\begin{proof} To simplify the notations, for the remainder of this proof, we will drop $|_{\vcomp}$ from $\overline{H_{V_{\rho}}|_{\vcomp}}$ and $\overline{H_{V}|_{\vcomp}}$ and write just $\overline{H_{V_{\rho}}}$ and $\overline{H_{V}}$. Let $T_1>T_0>0$, $R>0$, $\rho:=R+T_1$, and $f\in\vcomp$ with $\supp f\subset B_{R}$ be as in the hypothesis of the lemma. Define $U(t):=\cos\left((\overline{H_{V_{\rho}}})^{1/2}t\right)f$ with $t\in[0,T_0]$, and note that $U(t)$ satisfies the equation
\begin{equation}\label{E:2-3-or}
U''+\overline{H_{V_{\rho}}}U=0,\quad U(0)=f,\,\, U'(0)=0.
\end{equation}
Denoting $u(t,\cdot):=U(t)(\cdot)$ and using Lemma~\ref{L:2-2} with $W=V_{\rho}$ and $t\in[0,T_0]$, we have  $\supp u(t,\cdot)\subset B_{R+t}\subset B_{\rho_0}$, where $\rho_0:=R+T_0<R+T_1=\rho$.

As $U(t)\in \Dom(\overline{H_{V_{\rho}}})$ with $\supp u(t,\cdot)\subset B_{\rho_0}$, $\rho_0<\rho$, Lemma~\ref{L:2-2-5} tells us that
there exists a sequence $f_{j}(t,\cdot)\in\vcomp$ with with $\supp f_j(t,\cdot)\subset B_{\rho}$ such that
\begin{equation}\label{E:help-conv-1}
f_j(t,\cdot)\to U(t),\qquad H_{V_{\rho}}f_{j}(t,\cdot)\to \overline{H_{V_{\rho}}}U(t),
\end{equation}
for all $t\in [0,T_0]$, where both convergence relations are understood in the norm of $L^2(\vbn|_{B_{\rho}})$.

As $V_{\rho}f_{j}(t,\cdot)=Vf_{j}(t,\cdot)$, we can rewrite~(\ref{E:help-conv-1}) as
\begin{equation}\label{E:help-conv-2}
f_j(t,\cdot)\to U(t), \qquad H_{V}f_{j}(t,\cdot)\to \overline{H_{V_{\rho}}}U(t),
\end{equation}
for all $t\in [0,T_0]$.

Therefore, $U(t)\in \Dom(\overline{H_{V}})$ and $\overline{H_{V_{\rho}}}U=\overline{H_{V}}U$, for all $t\in [0,T_0]$. Recalling (see the hypothesis) that the operator $A$ is a self-adjoint extension of $H_{V}|_{\vcomp}$, we infer that $U(t)\in\Dom(A)$ and $\overline{H_{V}}U=AU$, for all $t\in [0,T_0]$. Thus, from~(\ref{E:2-3-or}) we get
\begin{equation}\label{E:2-4-or}
U''+AU=0,\quad U(0)=f,\,\, U'(0)=0,
\end{equation}
and from here, using Lemma~\ref{L:2-1}, we obtain
\begin{equation}\label{E:or-rep-76-1}
U(t)=\cos(A^{1/2}t)f,
\end{equation}
for all $t\in [0,T_0]$.

The argument used in~(\ref{E:help-conv-1}) and~(\ref{E:help-conv-2}) shows that $U(t)\in\Dom(F_{\rho})$ and
\begin{equation}\label{E:Friedrichs-rho-a}
\overline{H_{V_{\rho}}}U=F_{\rho}U,
\end{equation}
for all $t\in [0,T_0]$, where $F_{\rho}$ is the Friedrichs extension of $H_{V}|_{\vcompr}$ in $L^2(\vbn|_{B_{\rho}})$.

Going back to~(\ref{E:2-3-or}) and referring to~(\ref{E:Friedrichs-rho-a}) we get
\begin{equation}\nonumber
U''+F_{\rho} U=0,\quad U(0)=f,\,\, U'(0)=0.
\end{equation}
Finally, using Lemma~\ref{L:2-1} we obtain $U(t)=\cos(({F_{\rho}})^{1/2}t)f$ for all $t\in [0,T_0]$. Combining the latter formula with~(\ref{E:or-rep-76-1}) leads to~(\ref{E:formula-or}).
\end{proof}

\begin{remark}\label{rem-friedrichs} Following the notations of Lemma~\ref{L:2-3}, let $T_1>T_0>0$, $R>0$, $\rho_0:=R+T_0$, and $\rho:=R+T_1$. Furthermore, let $F_{\rho}$, $\rho>0$, be the Friedrichs extension of $H_{V}|_{\vcompr}$ in the space $L^2(\vbn|_{B_{\rho}})$. Taking into account the proof of Lemma~\ref{L:2-3}, we record the following observation: if $f\in\vcomp$  is a section such that $\supp f\subset B_{R}$ and if $U(t):=\cos(({F_{\rho}})^{1/2}t)f$, then $\supp [U(t)(\cdot)]\subset B_{\rho_0}\subset B_{\rho}$ for all  $t\in[0,T_0]$.
\end{remark}

\medskip

Before proceeding further, we explain some notations used below. If $S$ is a self-adjoint operator in a Hilbert space $\mathscr{H}$ with inner product $(\cdot,\cdot)$ and norm $\|\cdot\|$, then (see, for instance, Theorem B.5 in~\cite{Guneysu-2016}) $S$ has a unique spectral resolution $E(\lambda)$ such that $S=I_{\mathbb{R}}(E)$, where  $I_{\mathbb{R}}\colon \mathbb{R}\to \mathbb{R}$ is the identity function and $I_{\mathbb{R}}(E)$ is understood in the sense of spectral calculus. If $v\in\mathscr{H}$, then the function $\lambda\mapsto (E(\lambda)v,v)=\|E(\lambda)v\|^2$ is right-continuous and increasing, and, hence, gives rise to a Borel measure on $\mathbb{R}$, which we denote by $(E(d\lambda)v,v)$. For details on the spectral calculus, see section B.1.4 in~\cite{Guneysu-2016}.  

\medskip

\noindent\textbf{Continuation of the proof of Theorem~\ref{T:m-2}} Let $F_{\rho}$ be the Friedrichs extension of $H_{V}|_{\vcompr}$ in $L^2(\vbn|_{B_{\rho}})$. Looking at the hypothesis~(\ref{E:hyp-minorant-2}) and using $q(x):=[(\alpha\circ r)(x)]^2$, with $r(x)$ as in~(\ref{E:dist-x-0}), we see that
\begin{equation}
(H_{V} u,u)\geq -(\alpha(\rho))^2(u,u),
\end{equation}
for all $u\in \vcomp$ with $\supp u\subset B_{\rho}$, and, therefore,  $-(\alpha(\rho))^2$ is a lower bound of $F_{\rho}$.

In the subsequent discussion, $\cosh(\cdot)$ stands for the hyperbolic cosine. Denoting by $E_{F}(\lambda)$ the spectral resolution of $F_{\rho}$, we have,
\begin{align}\label{E:oroch-est-1}
&|(\cos(({F_{\rho}})^{1/2}t)v,v)|= \left|\int_{-\infty}^{\infty}\cos(\lambda^{1/2}t)\, (E_{F}(d\lambda)v,v)\right|\nonumber\\
&\leq \int_{-\infty}^{\infty}|\cos(\lambda^{1/2}t)|\, (E_{F}(d\lambda)v,v)=\int_{-(\alpha(\rho))^2}^{\infty}|\cos(\lambda^{1/2}t)|\, (E_{F}(d\lambda)v,v)\nonumber\\
&\leq \cosh (t\alpha(\rho))\|v\|_{L^2(\vbn|_{B_{\rho}})}^2,
\end{align}
for all $t\geq 0$ and all $v\in L^2(\vbn|_{B_{\rho}})$, where in the first equality we used the spectral calculus,  in the second equality we used the fact that $-(\alpha(\rho))^2$ is a lower bound of $F_{\rho}$, and in the second inequality we used the property
\[
|\cos(\lambda^{1/2}t)|\leq \cosh(\kappa^{1/2}t),
\]
which holds for all $t\geq 0$ and $\lambda\geq-\kappa$, with $\kappa>0$.
From~(\ref{E:oroch-est-1}) we obtain
\begin{equation}\label{E:oroch-est-2}
\|\cos(({F_{\rho}})^{1/2}t)\|_{2,2,\vbn|_{B_{\rho}}}\leq \cosh (t\alpha(\rho)),
\end{equation}
where $\|\cdot\|_{2,2,\vbn|_{B_{\rho}}}$ stands for the norm of a (bounded) operator $L^2(\vbn|_{B_{\rho}})\to L^2(\vbn|_{B_{\rho}})$.

Let $T_1>T_0>0$ and $R>0$ be arbitrary real numbers and let $\rho:=R+T_1$. Assume that there exists a self-adjoint extension $A$ of $H_{V}|_{\vcomp}$  in $L^2(\vbn)$, and let $E(\lambda)$ denote the spectral resolution of $A$. Keeping in mind the remark~\ref{rem-friedrichs}, for all sections $f\in\vcomp$  with $\supp f\subset B_{R}$ and all  $t\in[0,T_0]$,  we have
\begin{align}\label{E:oroch-est-3}
&\int_{-\infty}^{\infty}\cos^2(\lambda^{1/2}t)\,(E(d\lambda)f,f)=\|\cos(A^{1/2}t)f\|^2
=\|\cos(({F_{\rho}})^{1/2}t)f\|_{L^2(\vbn|_{B_{\rho}})}^2\nonumber\\
&\leq \|\cos(({F_{\rho}})^{1/2}t)\|_{2,2,\vbn|_{B_{\rho}}}^2\|f\|_{L^2(\vbn|_{B_{\rho}})}^2\leq \|f\|^2\cosh^2(t\alpha(\rho))\nonumber\\
&\leq \|f\|^2e^{2t\alpha(\rho)}=\|f\|^2e^{2t\alpha(T_1+R)},
\end{align}
where in the first equality we used the spectral calculus, in the second equality we used the representation~(\ref{E:formula-or}) and the remark~\ref{rem-friedrichs}, in the second inequality we used~(\ref{E:oroch-est-2}), and in the third estimate we used the definition of $\cosh(\cdot)$. Before proceeding further, we note that for all $\lambda<0$ and $t\geq 0$, we have
\begin{equation}\label{E:oroch-est-4}
\cos^2(\lambda^{1/2}t)=\cos^2(i|\lambda|^{1/2}t)=\cosh^2(|\lambda|^{1/2}t)\geq \frac{e^{2|\lambda|^{1/2}t}}{4},
\end{equation}
where $i$ is the imaginary unit.  Therefore, for all $t\in [0,T_0]$ we have
\begin{align}\label{E:oroch-est-5}
&Y(t):=\int_{-\infty}^{0}e^{2|\lambda|^{1/2}t}\,(E(d\lambda)f,f)\leq 4\int_{-\infty}^{0}\cos^2(\lambda^{1/2}t)(E(d\lambda)f,f)\nonumber\\
&\leq 4\int_{-\infty}^{\infty}\cos^2(\lambda^{1/2}t)\,(E(d\lambda)f,f)\leq 4\|f\|^2e^{2t\alpha(T_1+R)},
\end{align}
where in the first inequality we used~(\ref{E:oroch-est-4}) and in the third inequality we used~(\ref{E:oroch-est-3}). Using Cauchy--Schwarz inequality and~(\ref{E:oroch-est-5}) we obtain
\begin{align}\label{E:oroch-est-6}
&\int_{-\infty}^{0}e^{|\lambda|^{1/2}t}\,(E(d\lambda)f,f)\nonumber\\
&=\int_{-\infty}^{0}e^{|\lambda|^{1/2}t}\,[(E(d\lambda)f,f)]^{1/2}[(E(d\lambda)f,f)]^{1/2}\nonumber\\
&\leq (Y(t))^{1/2}\left(\int_{-\infty}^{0}(E(d\lambda)f,f)\right)^{1/2}\leq (Y(t))^{1/2}\|f\|\leq 2\|f\|^2e^{t\alpha(T_1+R)},
\end{align}
for all $t\in [0,T_0]$ and all $f\in\vcomp$  with $\supp f\subset B_{R}$. In particular, putting $t=T_0$ in~(\ref{E:oroch-est-6}) we get
\[
\int_{-\infty}^{0}e^{|\lambda|^{1/2}T_0}\,(E(d\lambda)f,f)\leq 2\|f\|^2e^{T_0\alpha(T_1+R)},
\]
which upon letting $T_1\to T_{0}+$ leads to
\[
\int_{-\infty}^{0}e^{|\lambda|^{1/2}T_0}\,(E(d\lambda)f,f)\leq 2\|f\|^2e^{T_0\alpha(T_0+R)},
\]
for all $f\in\vcomp$  with $\supp f\subset B_{R}$.

By the assumption on $\alpha(\cdot)$ and the properties of $\mathscr{V}$ mentioned before Theorem~\ref{T:m-2} it follows that $p(T_0):=T_0\alpha(T_0+R)$ belongs to $\mathscr{V}$. As $T_0>0$ and $R>0$ are arbitrary, we see that the measure $d\sigma_{f}(\lambda):= (E(d\lambda)f,f)$, where $f\in\vcomp$, satisfies the hypothesis~(\ref{E:vul-rep-2}) of Lemma~\ref{L:Vul-lemma}.

If $A_1$ and $A_2$ are two self-adjoint extensions of $H_{V}|_{\vcomp}$ in $L^2(\vbn)$ with the corresponding spectral resolutions
$E_1(\lambda)$ and $E_2(\lambda)$, then by Lemma~\ref{L:2-3} we have $\cos(A_{1}^{1/2}t)f=\cos(A_{2}^{1/2}t)f$, for all $t>0$ and all $f\in\vcomp$. Hence, by spectral calculus we have
\[
\int_{-\infty}^{\infty}\cos(\lambda^{1/2}t)\,(E_{1}(d\lambda)f,f)=\int_{-\infty}^{\infty}\cos(\lambda^{1/2}t)\,(E_{2}(d\lambda)f,f).
\]
Denoting the left-hand side (or the right-hand side) of the last equation by $Z(t)$ and keeping in mind that $d\sigma_{j,f}(\lambda):=(E_{j}(d\lambda)f,f)$, where $j=1,2$ and $f\in\vcomp$, satisfy the property~(\ref{E:vul-rep-2}), we can use Lemma~\ref{L:Vul-lemma} (and the fact that $E_{j}(\lambda)$ is a spectral resolution) to infer that $(E_{1}(\lambda)f,f)=(E_{2}(\lambda)f,f)$ for all $f\in\vcomp$ and all $\lambda\in\mathbb{R}$. As $\vcomp$ is dense in $L^2(\vbn)$, from the last equality we get $A_1=A_2$. Therefore, if the (densely defined, symmetric) operator $H_{V}|_{\vcomp}$ has a self-adjoint extension in $L^2(\vbn)$, then this self-adjoint extension is unique (that is, $H_{V}|_{\vcomp}$ is essentially self-adjoint). $\hfill\square$

\appendix
\section{Auxiliary lemmas}\label{S:aux-lemmas}
Before stating the following lemma (see Theorem 2.28 in~\cite{GT-book} for the proof), we explain some terminology. Let $\mathscr{H}$ be a Hilbert space with the inner product $(\cdot,\cdot)$. A conjugate linear map $\tau \colon\mathscr{H}\to \mathscr{H}$ is called a \emph{conjugation} if it satisfies $\tau^2 = I_{\mathscr{H}}$ and $(\tau u,\tau v) = (v,u)$, for all $u$, $v\in\mathscr{H}$.  An operator $S$ in $\mathscr{H}$ is called $\tau$-real if there exists a conjugation $\tau \colon\mathscr{H}\to \mathscr{H}$ such that $\tau (\Dom(S))\subset \Dom(S)$ and $S(\tau u) = \tau (Su)$, for all $u\in\Dom(S)$. An example of $\tau$ is provided by the usual complex conjugation in $\mathscr{H}=L^2(M)$, where $M$ is a Riemannian manifold and $L^2(M)$ is the space of square integrable complex-valued functions on $M$.

\begin{lemma}\label{L:GT-lemma} Assume that $S$ is a symmetric and $\tau$-real operator in a Hilbert space $\mathscr{H}$. Then, $S$ has a self-adjoint extension in $\mathscr{H}$.
\end{lemma}
Next,  we recall some known results used in section~\ref{S:m-2}. We start with Mazur's Lemma, whose statement can be found in Lemma 10.19 of~\cite{RR-book}. In this lemma, the symbol $\mathbb{Z}_{+}$ stands for the set $\{1,2,3,\dots\}$.
\begin{lemma}\label{L:Maz-L} Let $\mathscr{B}$ be a Banach space and suppose that $u_k\to f$ weakly in $\mathscr{B}$.  Then, there exists a function $G\colon \mathbb{Z}_{+}\to\mathbb{Z}_{+}$ and a sequence of sets of real numbers $\{(\beta(j))_{k}\}_{k=j}^{G(j)}$, with $(\beta(j))_{k}\geq 0$ and $\displaystyle \sum_{k=j}^{G(j)} (\beta (j))_{k} =1$, such that the sequence
\[
f_{j}:=\displaystyle \sum_{k=j}^{G(j)} (\beta (j))_{k}u_{k}
\]
converges strongly to $f$ in $\mathscr{B}$.
\end{lemma}

In the next result, whose proof can be found in~\cite{Vul-59}, the symbol $\mathscr{V}$ refers to the class of functions defined before the statement of Theorem~\ref{T:m-2}.

\begin{lemma} \label{L:Vul-lemma} Let $Z\colon (0,\infty)\to\mathbb{C}$ be a function representable as
\begin{equation}\label{E:vul-rep-1}
Z(t)=\int_{-\infty}^{\infty}\cos({\lambda}^{1/2}t)\,d\sigma(\lambda),
\end{equation}
for some function $\sigma\colon\mathbb{R}\to \mathbb{C}$ of bounded variation satisfying the following property: there exists a number $C$ and a function $\xi\in \mathscr{V}$ such that
\begin{equation}\label{E:vul-rep-2}
\int_{-\infty}^{0}e^{|\lambda|^{1/2}t}\,|d\sigma(\lambda)|\leq Ce^{\xi(t)},
\end{equation}
for all $t>0$. Then, the representation~(\ref{E:vul-rep-1}) is unique in the following sense: if $Z$ is representable in the form~(\ref{E:vul-rep-1}) with $\sigma=\sigma_1$ and $\sigma=\sigma_2$,  and if $\sigma_j$, $j=1,2$, satisfy the property~(\ref{E:vul-rep-2}), then $\sigma_1$ and $\sigma_2$ differ by a constant.
\end{lemma}
This result of~\cite{Vul-59} generalized an earlier result of~\cite{Levitan-Meiman}, where instead of the assumption $\xi\in\mathscr{V}$ in~(\ref{E:vul-rep-2}), a special case $\xi(t)=t^{\kappa}$, $0<\kappa\leq 2$, was considered.


\end{document}